\documentclass{amsart}

\newtheorem{theorem}{Theorem}[section]
\newtheorem{lemma}[theorem]{Lemma}
\newtheorem{cor}[theorem]{Corollary}
\theoremstyle{definition}
\newtheorem{definition}[theorem]{Definition}
\newtheorem{example}[theorem]{Example}
\theoremstyle{remark}

\usepackage{amsfonts}
\newcommand{\Fpbar}{\bar{\mathbb{F}}_p}
\newcommand{\Fp}{\mathbb{F}_p}
\newcommand{\B}{\mathcal{B}}
\newcommand{\fF}{{\mathfrak F}} 
\newcommand{\fH}{{\mathfrak H}} 
\newcommand{\fX}{{\mathfrak X}} 
\newcommand{\cser}{\mathcal{C}}
\newcommand{\nser}{\mathcal{N}}

\newcommand{\id}{\triangleleft}
\DeclareMathOperator{\Soc}{Soc}
\DeclareMathOperator{\pdef}{PDef}
\DeclareMathOperator{\Edef}{EDef}
\DeclareMathOperator{\Hom}{Hom}
\DeclareMathOperator{\Ext}{Ext}
\DeclareMathOperator{\sym}{{}^\text{sym}\mspace{-4mu}}
\DeclareMathOperator{\asym}{{}^\text{asym}\mspace{-4mu}}
\DeclareMathOperator{\nlen}{Nlen}
\DeclareMathOperator{\Leib}{Leib}

\begin{document}

\title{Conditions for a Schunck class to be a formation}

\author{Donald W. Barnes}
\address{1 Little Wonga Rd. Cremorne NSW 2090 Australia}
\email{donwb@iprimus.com.au}
\thanks{This work was done while the author was an Honorary Associate of the School
of Mathematics and Statistics, University of Sydney}

\subjclass[2010]{Primary 17B30; Secondary 17A32, 20D10}
\keywords{Lie algebras, formations}
\date{}

\begin{abstract} A Schunck class $\fH$ is determined by the class $\fX$ of primitives contained in $\fH$.  We give necessary and sufficient conditions on $\fX$ for $\fH$ to be a saturated formation.
\end{abstract}

\maketitle

\section{Introduction} \label{sec-Intro}
In \cite{Gasch}, Gasch\"utz introduced for finite soluble groups, the concepts of saturated formations and projectors.  He showed that if $\fF$ is a saturated formation, then $\fF$-projectors exist in every finite soluble group.  Saturated formations can be constructed via local definition, local in that for each prime $p$, there is a condition on the chief factors of $p$-power order for the group to be in the formation.

A class $\fH$ with the property that an $\fH$-projector exists in every finite soluble group is called a Schunck class.  Schunck showed that such a class $\fH$ is determined by the class $\fX$ of all primitive groups in $\fH$.  The class $\fX$ necessarily has the property that every primitive quotient of a group in $\fX$ is also in $\fX$.  Conversely, given such a class $\fX$, the class 
$$\fH = \pdef(\fX) = \{G \mid \text{every primitive quotient of }G\text{ is in }\fX\}$$
is a Schunck class, said to be primitively defined by $\fX$.  This theory is set out in detail in Doerk and Hawkes \cite{DH}.

Analogous theories have been developed for Lie algebras \cite{BGH}, restricted Lie algebras  \cite{Restricted}, and for Leibniz algebras \cite{SchunckLeib}.   In general outline, the theories are very similar, but while every saturated formation of soluble groups has many local definitions, a saturated formation of Lie algebras has at most one local definition.  (See \cite{local}.)  That not every saturated formation has a local definition means that we cannot use local definition to investigate all saturated formations of Lie algebras.  Every saturated formation, being a Schunck class, has a primitive definition.  To use this, we need to know when the Schunck class defined by the class $\fX$ of primitives is a formation.  The conditions on $\fX$ are essentially the same for groups, Lie algebras, restricted Lie algebras and Leibniz algebras.  So in the following, we treat all these together.  We use ``algebra'' to mean any of  group, Lie algebra, restricted Lie algebra or Leibniz algebra, all assumed to be finite or finite-dimensional and soluble.  Except for matters specific to groups, we write in the language of algebras.  A primitive is an algebra $P$ with a minimal ideal $K$ such that its centraliser $\cser_P(K)=K$.  In all cases, $P$ splits over $K$ and all the complements are conjugate.

In Section \ref{gps}, we establish for groups analogues of \cite[Lemmas 1.2,  1.7,  1.9, Theorems 2.1,  2.6]{Blocks}.  There is no need to do this for Leibniz algebras or for restricted Lie algebras.  The blocks for a restricted Lie algebra $(L,[p])$ are the same as for its underlying Lie algebra $L$.  

If $P$ is a primitive Leibniz algebra with $\Soc(P) = A$, then $P/A$ is a Lie algebra and, as $L/A$-module, $A$ is either symmetric, that is, $ax = -xa$ for $x \in P/A$ and $a \in A$, or asymmetric, that is, $ax=0$.   From the given left action on $A$ we can form the symmetric and asymmetric modules $\sym A$ and $\asym A$ and their split extensions by $P/A$, the primitive algebras $\sym P$ and $\asym P$, the given primitive algebra $P$ being one of these.  Thus primitive Leibniz algebras come in pairs $\{\sym P, \asym P \}$ with $\sym P$ a Lie algebra while $\asym P$ has Leibniz kernel $\Leib (\asym P) = \Soc(\asym P)$.  A saturated formation $\fF$ of Leibniz algebras containing one member of a pair also contains the other and is determined by  the Lie algebras in $\fF$.  See \cite[Theorem 3.16, Corollary 3.17]{SchunckLeib}.  Consequently, our main result for Leibniz algebras follows immediately from the result for Lie algebras.

In Section \ref{cdits}, we establish for all cases, the conditions on the class $\fX$ of primitives for $\pdef(\fX)$ to be a saturated formation.

\section{Finite groups}\label{gps}
If $A/B$ is a $p$-chief factor of the group $G$, it can be regarded as an $\Fp G$-module.  To avoid confusion between the multiplicative notation used for the group and additive notation for the module, we denote the module by $[A/B]$ and the module element corresponding to the element $a \in A/B$ by $[a]$.  Thus $[a_1a_2]= [a_1] + [a_2]$.  The action of $g \in G$ on $[a]$ is given by $g[a] = [gag^{-1}]$.

The results of this section do not need the full power of the assumption, required for their applications in the next section, that the group be soluble, so solubility is not assumed here.

\begin{lemma} \label{gp-sole} Suppose that $A$ of $p$-power order is the only minimal normal subgroup of the $p$-soluble group $G$.  Suppose that $G$ does not split over $A$.  Let $B/A$ be a minimal normal subgroup of $G/A$.  Then $B/A$ is a $p$-group and
\begin{enumerate}
\item If $B$ is not abelian, then $[A]$ is a quotient of $[B/A] \otimes [B/A]$,
\item If $B$ is abelian but not of exponent $p$, then $[A] \simeq [B/A]$.
\end{enumerate}
\end{lemma}

\begin{proof} If  $B/A$ is not a $p$-group, then $|B/A|$ is prime to $p$.  By the Schur-Zassenhaus Theorem, there exists a complement $U$ to $A$ in $B$, and, if $V$ is another complement, then $V = aUa^{-1}$ for some $a \in A$.  By the Frattini argument, $G$ is generated by $A$ and the normaliser $\nser_G(U)$ of $U$.  But $A \subseteq \Phi(G)$, so $\nser_G(U) = G$, contrary to $A$ being the only minimal normal subgroup of $G$.  Hence $B$ is a $p$-group.

Suppose that $B$ is not abelian.  Let $\bar{b}_i \in B/A$.  The map $[\bar{b}_1] \otimes [\bar{b}_2] \mapsto [b_1b_2b_1^{-1}b_2^{-1}]$ is a module homomorphism.  As $[A]$ is irreducible, $[A]$ is a quotient of $[B/A] \otimes [B/A]$.

Suppose that $B$ is abelian of exponent $p^2$.  Then the map $[\bar{b}] \mapsto [b^p]$ is an isomorphism.
\end{proof}

\begin{lemma} \label{qgpblock} Suppose that $N \id G$ and $V,W$ are irreducible $\Fp(G/N)$-modules in the same block.  Then $V, W$ are in the same $\Fp G$-block.
\end{lemma}

\begin{proof} 
There is a chain $V = V_0, V_1, \dots, V_n = W$ of irreducible $\Fp (G/N)$-modules and non-split extensions $X_i$ of either $V_{i-1}$ by $V_i$ or of $V_i$ by $V_{i-1}$ linking $V$ and $W$.  But the $V_i$ are irreducible $\Fp G$-modules and the $X_i$ are non-split $\Fp G$-modules linking $V$ to $W$ as $\Fp G$-modules.
\end{proof}

Denote the dual $\Hom(V, \Fp)$ of the module $V$ by $V^*$.  Denote the principal $\Fp G$-block by $\B_0(\Fp G)$.

\begin{lemma}\label{dual}  Suppose $V \in \B_0(\Fp G)$.   Then
$V^* \in \B_0(\Fp G)$. 
\end{lemma}

\begin{proof}  There exists a sequence $\Fp = A_0, A_1, \dots, A_n = V$
of irreducible modules and a sequence $X_1, \dots, X_n$ of non-split extensions $X_i$
either of $A_{i-1}$ by $A_i$ or of $A_i$ by $A_{i-1}$.  Dualising this gives a sequence
$F = A^*_0, A^*_1, \dots, A^*_n = V^*$ of irreducible modules and a sequence $X^*_1, \dots, X^*_n$ of non-split extensions $X^*_i$ either of
$A^*_i$ by $A^*_{i-1}$ or of  $A^*_{i-1}$ by $A^*_i$.
\end{proof} 

\begin{lemma} \label{diffK} Suppose that $A$ is a minimal normal subgroup of the $p$-soluble group $G$.  Let $V,W$ be irreducible $\Fp G$-modules.  Suppose that $A$ acts trivially on $V$ and non-trivially on $W$.  Then every extension of $V$ by $W$ or of $W$ by $V$ splits.
\end{lemma}

\begin{proof}
As $A$-module, $V$ is the direct sum of $\dim(V)$ copies of $\Fp$, while $W$ is the direct sum of conjugate non-trivial irreducible $A$-modules $W_i$.  Thus $\Hom(V,W)$ is a direct sum of non-trivial irreducible $A$-modules, so $H^0(A,\Hom(V,W)) = 0$.  If $A$ is not a $p$-group, then $|A|$ is prime to $p$ and $H^1(A, \Hom(V,W)) = 0$, so we may suppose that $A$ is a $p$-group.  Since $A$ is an abelian $p$-group acting non-trivially on the irreducible module $W_i$, we have $H^1(A, W_i) = 0$.  So again we have $H^n(A, \Hom(V,W)) = 0$ for $n=0,1$. By the Hochschild-Serre spectral sequence, it follows that $H^1(G, \Hom(V,W)) = 0$.  So every module extension of $W$ by $V$ splits.  As $\Hom(W,V)$ is a direct sum of copies of the duals of the $W_i$, similarly we have that every extension of $V$ by $W$ splits.   
\end{proof}

\begin{cor} \label{B0ker} Suppose that $A$ is a minimal normal subgroup of the $p$-soluble group $G$ and that $V$ is an irreducible $\Fp G$-module in the principal block.  Then $A$ acts trivially on $V$.
\end{cor}

\begin{proof} As $A$ acts trivially on $\Fp$, it follows by Lemma \ref{diffK}, that $A$ acts trivially on every irreducible in a chain linking $\Fp$ to $V$.
\end{proof}

\begin{lemma} \label{compl} Let $A/B$ be a complemented $p$-chief factor of the $p$-soluble group $G$.  Then $[A/B] \in \B_0(\Fp G)$.
\end{lemma}

\begin{proof}  By \cite[Theorem 1]{H1G}, $\Ext^1_{\Fp G}(\Fp,V) = H^1(G,V) \ne 0$.
\end{proof}

\begin{theorem} \label{chiefsB0} Let $A/B$ be a $p$-chief factor of the $p$-soluble group $G$.  Then $[A/B] \in \B_0(\Fp G)$.
\end{theorem}

\begin{proof}  The result holds trivially if $|G|=p$.  We use induction over $|G|$.  By Lemma \ref{qgpblock}, we may suppose that $B = 1$ and that $[A] \notin \B_0(\Fp(G/A)$.  But then $H^n(G/A,A)=0$ for all $n$ and $A$ is a complemented $p$-chief factor. The result follows by Lemma \ref{compl}.
\end{proof}

\begin{lemma} \label{eval} Let $V,W$ be $\Fp G$-modules.  Then the evaluation map $$\epsilon: V \otimes \Hom(V,W) \to W$$
given by $\epsilon(v\otimes f) = f(v)$ is a module homomorphism.
\end{lemma}

\begin{proof} For $x \in G$, $\epsilon x(v \otimes f) = \epsilon(xv \otimes  xf) = (xf)(xv) =  xf(x^{-1}xv) = x \epsilon(v \otimes f)$.
\end{proof}

\begin{theorem} \label{tens} Suppose $V,W$ are irreducible $\Fp G$-modules and that there exists a non-split extension of $W$ by $V$.  Then $W$ is a quotient of $V \otimes A$ for some $A \in \B_0(\Fp G)$.
\end{theorem}

\begin{proof} Since $H^1(G, \Hom(V,W)) \ne 0$, $\Hom(V,W)$ must have some composition factor in $\B_0(\Fp G)$.  Let $B$ be the $\B_0$-component of $\Hom(V,W)$ in its block decomposition.  Then $B \ne 0$.  Take a minimal submodule $A \subseteq B$.  Then $\epsilon(V\otimes A) \ne 0$, so $\epsilon(V \otimes A) = W$. 
\end{proof}

\begin{theorem} \label{chiefs}  Let $C$ be the set of the $p$-chief factor modules of the $p$-soluble group $G$ and their duals.  Let $V$ be an irreducible $\Fp G$-module in $\B_0(\Fp G)$.  Then $V$ is a composition factor of some tensor product $C_1 \otimes \dots \otimes C_k$ of modules $C_i \in C$.
\end{theorem}

\begin{proof} By induction over the length of the sequence linking $V$ to $F$, we may suppose that we have a non-split extension $X$ of $V$ by $W$ or of $W$ by $V$ with $W$ a composition factor of some tensor product of modules in $C$.  Since for modules $M,N$, $(M \otimes N)^* \simeq M^* \otimes N^*$, by Lemma \ref{dual}, we need only consider the case where  $X$ is a non-split extension of $V$ by $W$. 

Let $A$ be a minimal normal subgroup of $G$.  By Corollary \ref{B0ker}, $V$ and $W$ are $\Fp( G/A)$-modules.  If $X$ also is an $\Fp(G/A)$-module, then $V,W$ are in the same $\Fp(G/A)$-block and by Theorem \ref{tens}, $V$ is a quotient of $B \otimes W$ for some $B \in \B_0(L/A)$.  But by induction over $\dim(L)$, $B$ is a composition factor of some tensor product of $p$-chief factor modules of $G/A$ and their duals.  Thus the result holds in this case.

Now suppose that no non-split $\Fp(G/A)$-module extension of $V$ by $W$ exists.  Then $H^1(G/A, \Hom(W,V))=0$.  That is, $H^1(G/A, \Hom(W,V)^A) = 0$ as $A$ acts trivially on $\Hom(W,V)$.  But $X$ is a non-split $\Fp G$-module extension of $V$ by $W$, so  $H^1(G,\Hom(W,V)) \ne 0$.  By the Hochschild-Serre spectral sequence,  and we must have $H^1(A,\Hom(W,V))^G \ne 0$, so $A$ cannot have order prime to $p$.  So $A$ is an abelian $p$-group which acts trivially on $\Hom(W,V)$. Therefore $H^1(A, \Hom(W,V)) = \Hom([A], \Hom(W,V))$ and it follows that we have a non-zero $\Fp G$-module homomorphism $f: A \to \Hom(W,V)$.  Then $f(A)$ is a nonzero submodule of $\Hom(W,V)$ and the evaluation map $\epsilon$ maps $f(A) \otimes W$ onto $V$.  The result follows.
\end{proof}

\section{The conditions} \label{cdits}
Let $\fX$ be a class of primitive algebras.
\begin{definition} We say that $\fX$ is \textit{primitive quotient closed} if, for every $P \in \fX$, also every primitive quotient of $P$ is in $\fX$.
\end{definition}
That $\fX$ is primitive quotient closed is necessary and sufficient for $\pdef(\fX)$ to be a Schunck class.  If $\fF = \pdef(\fX)$ is a formation, then for every chief factor $A/B$ of $P$, we must have also the split extension $Q$ of $A/B$ by $P/\cser_P(A/B) \in \fX$.  
\begin{definition}
If, for every $P \in \fX$ and every chief factor $A/B$ of $P$, the split extension of $A/B$ by $P/\cser_P(A/B)$ is in $\fX$, we say that $\fX$ is \textit{chief factor closed}.
\end{definition}

If $\fX$ is chief factor closed, then clearly, it is primitive quotient closed.

If $\fF = \pdef(\fX)$ is a saturated formation, then the dual of an $\fF$-central module is $\fF$-central.  Thus for $P \in \fX$ with $A = \Soc(P)$, we must have that the split extension of the dual $\Hom(A,F)$ of $A$ by $P/A$ is in $\fX$.  

\begin{definition} We say that $\fX$ is \textit{dual closed} if, for every $P \in \fX$, the split extension of the dual $\Hom(A,F)$ of $A = \Soc(P)$ by $P/A$ is in $\fX$.  
\end{definition}

\begin{definition}  (For Leibniz algebras.  The condition is meaningless and to be regarded as always satisfied in the other cases.)  We say that $\fX$ is \textit{paired} if, for every $P \in \fX$, both members of the pair $(\sym P, \asym P)$ are in $\fX$.
\end{definition}

Let $P,Q \in \fX$ with $A = \Soc(P)$ and $B=\Soc(Q)$.  Let $L$ be a subdirect sum of $P/A$ and $Q/B$.  Then $A,B$ are $L$-modules.  Let $C$ be a composition factor of $A\otimes B$ and let $R$ be the split extension of $C$ by $L/\cser_L(C)$.  We call $R$ a \textit{subtensor product} of $P$ and $Q$.  If $\fF$ is a saturated formation, then $A \otimes B$ is an $\fF$-hypercentral $L$-module and so we must have $R \in \fX$.  

\begin{definition}If, for all $P,Q \in \fX$, every primitive subtensor product $R$ of $P$ and $Q$ is in $\fX$, we say that $\fX$ is \textit{subtensor closed}.
\end{definition}

\begin{definition}  Let $\fX$ be a class of primitive algebras.  We say that the chief factor $A/B$ of $L$ is $\fX$-central if the split extension of $A/B$ by $L/\cser_L(A/B)$ is in $\fX$.
\end{definition}

\begin{lemma} \label{subtens}  Suppose that $\fX$ is a chief factor, dual and subtensor closed and paired class of primitive algebras.  Let $\fF$ be the class of algebras all of whose chief factors are $\fX$-central.  Then $\fF$ is the saturated formation $\pdef(\fX)$.
\end{lemma}

\begin{proof}  If the result holds for Lie algebras, then by   \cite[Theorem 3.16, Corollary 3.17]{SchunckLeib}, the result holds for Leibniz algebras.  So we need only prove the result for the other three categories.   Since $\fF$ is defined in terms of chief factors, $\fF$ is a formation.  We have to prove that it is saturated.  Suppose that $A$ is a minimal ideal of the algebra $L$, that $L/A \in \fF$ and that $L$ does not split over $A$.  Then $A \subseteq \Phi(L)$.  We have to prove that $A$ is $\fX$-central.  We use induction over $\dim(L)$. The result holds trivially if $A = L$. Suppose that $B$ is another minimal ideal of $L$.  Then $A+B/B$ is a minimal ideal of $L/B$ and $A/B \subseteq \Phi(L/B)$.  By induction, $A+B/B$ is $\fX$-central.  Thus we may suppose that $A$ is the only minimal ideal of $L$.

Let $B/A$ be a minimal ideal of $L/A$.  Since $A \subseteq \Phi(L)$, it follows that $B$ is nilpotent.  Suppose that $B$ is not abelian.  Then $B' = A$ and we have an epimorphism $\epsilon: B/A \otimes B/A \to A$ defined by $\bar{b}_1 \otimes \bar{b}_2 \mapsto b_1b_2$ (for groups,  $\bar{b}_1 \otimes \bar{b}_2 \mapsto b_1b_2b_1^{-1}b_2^{-1}$ by Lemma \ref{gp-sole}).  Since $\fX$ is subtensor closed, the split extension of $A$ by $L/\cser_L(A)$ is in $\fX$, that is, $A$ is $\fX$-central. 

Now suppose that $B$ is abelian.  Then $B$ is an $L/B$-module which does not split over the submodule $A$.    By \cite[Theorem 1.5]{Blocks} (Lemma \ref{gp-sole}  and Theorem \ref{tens} for groups), $A$ is a quotient of $V \otimes (B/A)$ for some $V$ in the principal block of $L/B$.  But by \cite[Theorem 2.6]{Blocks} (Theorem \ref{chiefs}  for groups), $V$ is a composition factor of a tensor product of chief factors of $L/B$ and their duals.  From the closure properties of $\fX$, it follows that $A$ is $\fX$-central.

Now let $\fH = \pdef(\fX)$.  Then $\fH$ is the class of algebras all of whose {\em complemented} chief factors are $\fX$-central.  Therefore $\fF \subseteq \fH$.  Suppose that $\fF \ne \fH$.  Then we can take $L \in \fH$, $L \notin \fF$ of least possible dimension.  Let $A$ be a minimal ideal of $L$.  Then $L/A \in \fF$ and every chief factor of $L/A$ is $\fX$-central.  But $\fF$ is saturated, so $L$ splits over $A$, $A$ is a complemented chief factor of $L \in \fH$, so $A$ is $\fX$-central.  Therefore $L \in \fF$ contrary to the choice of $L$.
\end{proof}

Denote the nilpotent length of $L$ by $\nlen(L)$.  For groups, we use the $p$-nilpotent length denoted by $p$-$\nlen(L)$.

\begin{theorem} Suppose that $\fX$ is a primitive quotient, dual and subtensor closed and paired class of primitive algebras.  Let $\fF$ be the class of algebras all of whose chief factors are $\fX$-central.  Then $\fF$ is the saturated formation $\pdef(\fX)$.
\end{theorem}

\begin{proof} If $\fX$ is chief factor closed, then the result holds by Lemma \ref{subtens}, so let $P \in \fX$ be a primitive algebra with a chief factor $A/B$ which is not $\fX$-central.  We take $P$ with $n = \nlen(P)$ least possible.  (For groups, we choose $P$ and $p$ to make $n=p$-$\nlen(P)$ as small as possible.)  Let $\fX_0$ be the class of algebras $L \in \fX$ with $\nlen(L) < n$.  (For groups, with $p$-$\nlen(L) < n$ for all $p$.)  Then $\fX_0$ is chief factor, dual and subtensor closed and paired.  By Lemma \ref{subtens}, $\pdef(\fX_0)$ is a saturated formation.  But $P/\Soc(P) \in \pdef(\fX_0)$, so every chief factor of $P/\Soc(P)$ is $\fX_0$-central, contrary to hypothesis.
\end{proof}

\begin{cor}  Let $\fH$ be a Schunck class and let $\fX$ be the class of primitive algebras in $\fH$.  Suppose that $\fX$ is  dual and subtensor closed and paired.  Then $\fH $  is a saturated formation.
\end{cor}

\begin{proof} Let $\fF$ be the class of algebras whose chief factors are all $\fX$-central.  As $\fH$ is the class of algebras whose complemented chief factors are $\fX$-central,  $\fF \subseteq \fH$.  We prove that if $L \in \fH$, then $L \in \fF$.  Let $L$ be a minimal counterexample. and let $A$ be a minimal ideal of $L$.  Then $L/A \in \fF$.  But $\fF$ is saturated.  Since $L \notin \fF$, $L$ splits over $A$.  But every complemented chief factor of $L$ is $\fX$-central, so $A$ is $\fX$-central and $L \in \fF$ contrary to assumption.
\end{proof}

\begin{example}  If $\fX$ is the class of primitive algebras  $P$ with $\dim(\Soc(P)) = 1$, then $\fX$ is easily seen to satisfy the conditions.  The class $\pdef(\fX)$ is the class of supersoluble algebras.  (For restricted Lie algebras, the chief factors are either $1$-dimensional or central with the $p$-operation acting invertibly.) 
\end{example}

\begin{example}  Let $\Lambda$ be a normal $F$-subspace of the algebraic closure $\bar{F}$ of $F$.  For groups, $\Lambda_p$ either $\emptyset$ or a subgroup of the multiplicative group of $\Fpbar$.  Let $\fX$ be the class of primitives $P$ for which, for all $x \in P$ and all chief factors $A/B$ of $P$, all eigenvalues of the action of $x$ on $A/B$ are in $\Lambda$.  Then $\fX$ satisfies the conditions and it follows that the class $\Edef(\Lambda)$ of all algebras $L$ such that, for all $x \in L$ and all chief factors $A/B$ of $L$, all eigenvalues of the action of $x$ on $A/B$ are in $\Lambda$ is a saturated formation.
\end{example}

\bibliographystyle{amsplain}

\end{document}